\documentclass[12pt]{amsart}
\usepackage{amsmath}
\usepackage{amssymb}
\usepackage{mathrsfs}
\usepackage[top=3cm,bottom=3.5cm,right=2.5cm,left=2.5cm]{geometry}
\usepackage{ulem}
\newcommand{\Z}{\mathbb{Z}}
\newcommand{\Q}{\mathbb{Q}}
\newcommand{\N}{\mathbb{N}}
\newcommand{\R}{\mathbb{R}}
\newcommand{\T}{\mathbb{T}}
\newcommand{\C}{\mathbb{C}}
\newcommand{\dist}{{\rm dist}}
\newcommand{\Dist}{{\rm Dist}}
\newcommand{\diam}{{\rm diam}}

\usepackage{amsthm}
\newtheorem{dfn}{Definition}[section]
\newtheorem{prop}{Proposition}[section]
\newtheorem{conj}[prop]{Conjecture}
\newtheorem{thm}[prop]{Theorem}
\newtheorem{cor}[prop]{Corollary}
\newtheorem{lem}[prop]{Lemma}
\theoremstyle{definition}

\usepackage{graphicx}

\title[$\times a$ and $\times b$ empirical measures, the irregular set and entropy]{$\boldsymbol{\times} \boldsymbol{a}$ and $\boldsymbol{\times} \boldsymbol{b}$ empirical measures, the irregular set and entropy}
\author{Shunsuke Usuki}
\date{}

\begin{document}
	\maketitle
	
	\begin{abstract}
		For integers $a$ and $b\geq 2$, let $T_a$ and $T_b$ be multiplication by $a$ and $b$ on $\T=\R/\Z$.
		The action on $\T$ by $T_a$ and $T_b$ is called $\times a,\times b$ action and
		it is known that, if $a$ and $b$ are multiplicatively independent, then the only $\times a,\times b$ invariant and ergodic measure with positive entropy of $T_a$ or $T_b$ is the Lebesgue measure. However, whether there exists a nontrivial $\times a,\times b$ invariant and ergodic measure is not known. 
		In this paper, we study the empirical  measures of $x\in\T$ with respect to the $\times a,\times b$ action and
		show that the set of $x$ such that the empirical measures of $x$ do not converge to any measure has Hausdorff dimension $1$ and the set of $x$ such that the empirical measures can approach a nontrivial $\times a,\times b$ invariant measure has Hausdorff dimension zero. Furthermore, we obtain some equidistribution result about the $\times a,\times b$ orbit of $x$ in the complement of a set of Hausdorff dimension zero.
	\end{abstract}
	
	\section{Introduction and main theorems}
	In this paper, we write $\Z_{\geq 0}$ for the set of integers equal to or larger than $0$ and $\N$ for the set of positive integers.
	Let $\T=\R/\Z$ and, for $a\in\Z$ with $a\geq2$, define $T_a:\T\to\T$ by
	$$T_a(x)=ax,\quad x\in\T.$$
	We take $a,b\in\Z$ such that $a,b\geq2$. Since $T_a$ and $T_b$ are commutative, they define the $\Z_{\geq0}^2$-action on $\T$ and we call it the 
	{\bf $\boldsymbol{\times a,\times b}$ action}. Here we notice that, if $\log a/\log b\in\Q$, then $a=c^k$ and $b=c^l$ for some $c\geq2, k,l\in\N$, and the $\times a,\times b$ action derives from the $\times c$ action by the single map $T_c$.
	Therefore we are interested in the case that $a$ and $b$ are {\bf multiplicatively independent}, that is, $\log a/\log b\notin\Q$.
	
	There is the distinction between the $\times a$ action by the single map $T_a$ and the $\times a,\times b$ action about the closed invariant subsets.
	It is well-known that the $\times a$ action has many invariant closed subset of $\T$. However, H. Furstenberg showed that $\times a,\times b$ invariant, that is, invariant under $T_a$ and $T_b$, closed subsets are very restricted.
	
	\begin{prop}[{\cite[Theorem IV.1]{Fur}}]\label{invariantsubset}
		Suppose $a$ and $b$ are multiplicatively independent, that is, $\log a/\log b\notin \Q$.
		Let $X\subset\T$ be a nonempty, closed and $\times a,\times b$ invariant subset. Then $X=\T$ or $X$ is a finite set in $\Q/\Z$.
	\end{prop}
	
	He also conjectured the measure-theoretic version of Proposition \ref{invariantsubset}. 
	We write $M(\T)$ for the set of Borel probability measures on $\T$ and $M_{\times a,\times b}(\T)$ for the set of $\times a,\times b$ invariant Borel probability measures on $\T$, that is, the set of $\mu\in M(\T)$ such that $\mu$ is invariant under $T_a$ and $T_b$. Furthermore, we write $E_{\times a,\times b}(\T)$ for the set of $\times a,\times b$ invariant and ergodic probability measures on $\T$, that is, the set of $\mu\in M_{\times a,\times b}(\T)$ such that $\mu$ is ergodic with respect to the $\Z_{\geq 0}^2$-action by $T_a$ and $T_b$.
	The Lebesgue measure on $\T$ is denoted by $m_\T$. We notice that $m_\T\in E_{\times a,\times b}(\T)$.
	
	\begin{conj}\label{FurConj}
		Suppose $a$ and $b$ are multiplicatively independent. Let $\mu\in E_{\times a,\times b}(\T)$. Then $\mu=m_\T$ or $\mu$ is an atomic measure equidistributed on a $\times a,\times b$ peridic orbit on $\Q/\Z$.
	\end{conj}
	
	This problem is open for a long time. However, the following theorem was shown by D.J. Rudolph in \cite{Rud} when $a$ and $b$ are relatively prime and by A.S.A. Johnson in \cite{Joh} when $a$ and $b$ are
	multiplicatively independent. For a $T$-invariant probability measure $\mu$ ($T=T_a$ or $T_b$),
	we write $h_{\mu}(T)$ for the measure-theoretic entropy of $T$ with respect to $\mu$.
	\begin{thm}[The Rudolph-Johnson Theorem]\label{RudolphJohnson}
		Suppose $a$ and $b$ are multiplicatively independent. Let $\mu\in E_{\times a,\times b}(\T)$ such that $h_\mu(T_a)>0$ or $h_\mu(T_b)>0$. Then $\mu=m_{\T}$.
	\end{thm}
	
	By Theorem \ref{RudolphJohnson}, if there exists some nontrivial $\times a,\times b$ invariant and ergodic probability measure $\mu$, then $h_\mu(T_a)=h_\mu(T_b)=0$. There are distinct proofs of Theorem \ref{RudolphJohnson} and stronger results in  \cite{Fel93}, \cite{Hos95} and \cite{Hoc12}, though the positive entropy assumption is crucial in all of them.
	
	For $x\in\T$, let $\delta_x$ be the probability measure supported on the one point set $\{x\}$.
	For each $N\in\N$, we write $\delta_{\times a,\times b,x}^N\in M(\T)$ for the {\bf $N$-empirical measure of $\boldsymbol{x}$ (with respect to the $\boldsymbol{\times a,\times b}$ action)}, that is,
	$$\delta_{\times a,\times b,x}^N=\frac{1}{N^2}\sum_{m,n=0}^{N-1}\delta_{T_a^mT_b^nx}.$$
	If we give $M(\T)$ the weak* topology, then $M(\T)$ is a compact and metrizable space. It is easily seen that any accumulation point in $M(\T)$ of $\delta_{\times a,\times b,x}^N\ (N\in\N)$, that is, $\mu\in M(\T)$ such that $\delta_{\times a,\times b,x}^{N_k}\to\mu$ in $M(\T)$ as $k\to\infty$ for some divergent subsequence $\{N_k\}_{k=1}^\infty$ in $\N$, is $\times a,\times b$ invariant.
	If $\mu\in E_{\times a,\times b}(\T)$, then, by Birkhoff's ergodic theorem, 
	$$\delta_{\times a,\times b,x}^N\xrightarrow[N\to\infty]{}\mu,\text{   for $\mu$-a.e }x.$$
	We refer \cite{Kel98} for Birkhoff's ergodic theorem for $\Z_{\geq 0}^2$-actions.
	In this paper we study two types of subsets of $\T$ about the behavior of $\delta_{\times a,\times b,x}^N$ as $N\to\infty$: the set of $x$ such that $\delta_{\times a,\times b,x}^N$ does not converge to any invariant measure, which is called the {\bf irregular set} for the empirical measure, and the set of $x$ such that $\delta_{\times a,\times b,x}^N$ accumulates to some invariant probability measure which has the given upper bound of entropy. Our main results give estimate of Hausdorff dimension of these sets.
	
	We give the first main result in this paper about the irregular set. We write $J$ for the irregular set. We notice that, by Birkhoff's ergodic theorem, $\mu(J)=0$ for any $\mu\in M_{\times a,\times b}(\T)$. However, in general, the irregular set can be either small or large. For example, it is clear that, if an action on a compact metric space is uniquely ergodic, then its irregular set is empty. 
	On the other hand, the following fact holds for the $\times a$ action by the single map $T_a$. For a Hölder continuous function $\varphi: \T\to\R$, we write $J_\varphi$ for the irregular set for $\varphi$, that is, the set of $x\in\T$ such that the Birkhoff average $N^{-1}\sum_{n=0}^{N-1}\varphi(T_a^nx)\ (N\in\N)$ does not converge as $N\to\infty$. If $\varphi$ is not cohomologous to a constant, then $\dim_H J_\varphi=1$ and hence the irregular set for the empirical measure has Hausdorff dimension $1$. We remark that this fact holds under more general situations (see \cite{BS00}). Under these situations, there exist many distinct invariant and ergodic measures which have sufficiently large dimension, and hence many subsets with large Hausdorff dimension on which the Birkhoff average converges to distinct values.
	Since $\times a,\times b$ invariant and ergodic measures on $\T$ are restricted by Theorem \ref{RudolphJohnson}, the situation of the $\times a,\times b$ action is different from that we mentioned above.
	However, it is shown that the irregular set is a subset of $\T$ with large Hausdorff dimension.
	In \cite{FQQ21}, it is shown that the set of $x\in\T$ such that the $\times a,\times b$ empirical measures do not converge to $m_\T$ has positive Hausdorff dimension\footnote{In \cite{FQQ21}, $a$ and $b$ are restricted to $a=2$ and $b=3$ and the way of taking averages is different. We can do the same argument on this averages.}.
	Our theorem below is a stronger result.
	
	\begin{thm}\label{doesnotconverge}
		Let $J$ be the set of $x\in\T$ such that $\delta_{\times a,\times b,x}^N\ (N\in\N)$ does not converge to any $\times a,\times b$ invariant probability measure as $N\to\infty$.
		Then
		$$\dim_HJ=1.$$
	\end{thm}
	We notice that Theorem \ref{doesnotconverge} is shown without the hypothesis that $a$ and $b$ are multiplicatively independent.
	It is remarkable that Theorem \ref{doesnotconverge} immediately leads the following stronger result than itself, which is about the irregular sets for Fourier basis functions. For $k\in\Z$, we write $e_k(x)=e^{2k\pi i x}\ (x\in\T)$ and, as above, $J_{e_k}$ for the irregular set for $e_k$, that is, the set of $x\in\T$ such that the Birkhoff average $N^{-2}\sum_{m,n=0}^{N-1}e_k(T_a^mT_b^nx)\ (N\in\N)$ does not converge as $N\to\infty$.
	
	\begin{cor}\label{irregular_fourier}
		For $k\in\Z\setminus\{0\}$, we have
		$$\dim_HJ_{e_k}=1.$$
	\end{cor}
	We prove Theorem \ref{doesnotconverge} and Corollary \ref{irregular_fourier} in Section \ref{proofdoesnotconverge}.
	
	Next we give the second main result. As we said above, if $a$ and $b$ are multiplicatively independent, it is conjectured that there exist no nontrivial $\times a,\times b$ invariant and ergodic measures (Conjecture \ref{FurConj}). This problem seems to be very difficult, however, by Theorem \ref{RudolphJohnson}, those nontrivial invariant measures have entropy zero.
	We expect that the set of $x\in\T$ such that $\delta_{\times a,\times b,x}^N$ approaches a nontrivial measure as $N\to\infty$ is a small subset of $\T$. The following theorem and corollary answer this expectation. 
	
	\begin{thm}\label{entropy0}
		Let $0<t< \min\{\log b,(\log a)^2/\log b\}$ and $K_t$ be the set of $x\in\T$ such that $\delta_{\times a,\times b,x}^N\ (N\in\N)$ accumulates to some $\mu\in M_{\times a,\times b}(\T)$ such that $h_{\mu}(T_a)\leq t$.
		Then
		$$\dim_HK_t\leq\frac{2\sqrt{\log b}\sqrt{t}}{\log a+\sqrt{\log b}\sqrt{t}}.$$
	\end{thm}
	
	We notice that Theorem \ref{entropy0} is shown without the hypothesis that $a$ and $b$ are multiplicatively independent. By taking $\bigcap_{t>0}K_t$ and applying Theorem \ref{RudolphJohnson}, we obtain the following corollary.
	
	\begin{cor}\label{entropy00}
		Suppose $a$ and $b$ are multiplicatively independent.
		Let $K$ be the set of $x\in\T$ such that $\delta_{T_a,T_b,x}^N\ (N\in\N)$ accumulates to some $\mu\in E_{T_a,T_b}(\T)$ such that $\mu\neq m_\T$.
		Then
		$$\dim_HK=0.$$
	\end{cor}
	
	If $a$ and $b$ are multiplicatively independent, Theorem \ref{entropy0} and Theorem \ref{RudolphJohnson} lead the result about the distributions of the $\times a,\times b$ orbits.
	For $0<t\leq 1$ and $x\in\T$, we say that the $\times a,\times b$ orbit $\{a^mb^nx\}_{m,n\in\Z_{\geq 0}}$ of $x$ is {\bf $\boldsymbol{t}$-semiequidistributed} if 
	$$\liminf_{N\to\infty}\frac{1}{N^2}\sum_{m,n=0}^{N-1}f(a^mb^nx)\geq t\int_\T f\ dm_\T$$
	for any $f\in C(\T)$ such that $f\geq 0$ on $\T$ and
	$$\liminf_{N\to\infty}\frac{1}{N^2}\left|\left\{(m,n)\in\Z^2\left|\ 0\leq m,n<N, a^mb^nx\in U\right.\right\}\right|\geq t\cdot m_\T(U)$$
	for any open subset $U\subset\T$. 
	It is easy to see that the latter statement follows from the former. 
	This property says that the orbit $\{a^mb^nx\}_{m,n\in\Z_{\geq 0}}$ includes an equidistributed portion of the ratio at least $t$.
	Then we have the following.
	
	\begin{thm}\label{distribution}
		Suppose $a$ and $b$ are multiplicatively independent. Let $0<t<$\\ $\min\{\log b,(\log a)^2/\log b\}$ and $K_t\subset\T$ be as above. Then, for each $x\in\T\setminus K_t$, the orbit $\{a^mb^nx\}_{m,n\in\Z_{\geq 0}}$ is $t/\log a$-semiequidistributed.
	\end{thm}
	
	If $t>0$ is small, by Theorem \ref{entropy0}, we have that $\dim_H K_t\leq O(\sqrt{t})$ and Theorem \ref{distribution} implies that, for $x\in\T$, the orbit $\{a^mb^nx\}_{m,n\in\Z_{\geq 0}}$ is $t/\log a$-semiequidistributed if $x$ is in the complement of the set of small Hausdorff dimension about $\sqrt{t}$. 
	In particular, by taking $X=\bigcup_{t>0}(\T\setminus K_t)$, we have the following corollary.
	
	\begin{cor}
		Suppose $a$ and $b$ are multiplicatively independent. Then there exists $X\subset\T$ such that $\dim_H(\T\setminus X)=0$ and, for any $x\in X$, the $\times a,\times b$ orbit $\{a^mb^nx\}_{m,n\in\Z_{\geq 0}}$ of $x$ is $s$-semiequidistributed for some $s=s(x)>0$.
	\end{cor}
	
	We notice that the $\times a$ action on $\T$ by the single $T_a$ does not exhibit this property, since there exists a $\times a$ invariant Cantor set $C\subset\T$ such that $0<\dim_H C<1$.
	We will prove Theorem \ref{entropy0} and \ref{distribution} in Section \ref{proofentropy}.
	
	\section{Proof of Theorem \ref{doesnotconverge} and Crollary \ref{irregular_fourier}}\label{proofdoesnotconverge}
	
	In this section, we prove Theorem \ref{doesnotconverge} and Corollary \ref{irregular_fourier}.
	First, we see that Theorem \ref{doesnotconverge} leads immediately Corollary \ref{irregular_fourier}.
	
	\begin{proof}[Proof of Corollary \ref{irregular_fourier}]
		We assume that Theorem \ref{doesnotconverge} holds. Since the linear space spanned by $\{e_k\}_{k\in\Z}$ over $\C$ is dense in the Banach space of $\C$-valued continuous functions on $\T$ with the supremum norm and $J_{e_0}=\emptyset$, it can be seen that $J=\bigcup_{k\in\Z\setminus\{0\}}J_{e_k}$. Hence, using Theorem \ref{doesnotconverge}, we have
		\begin{equation}\label{sup_dim_jek}
			1=\dim_HJ=\sup_{k\in\Z\setminus\{0\}}\dim_HJ_{e_k}
		\end{equation}
		For $k\in\Z\setminus\{0\}$, $T_k:\T\ni x\mapsto kx\in\T$ is commutative with $T_a$ and $T_b$ and $e_k=e_1\circ T_k$. Therefore we have $J_{e_k}=T_k^{-1}J_{e_1}$.
		Moreover, it can be seen that $\dim_HT_k^{-1}J_{e_1}=\dim_H J_{e_1}$. From these and the equation (\ref{sup_dim_jek}), it follows that
		$$
		1=\dim_H J_{e_1}=\dim_H J_{e_k}
		$$
		and we complete the proof.
	\end{proof}
	
	Next, we prove Theorem \ref{doesnotconverge}. We develop the method in \cite{FQQ21} and construct subsets of $J$ which have Hausdorff dimension arbitrarily near $1$. We need the notion of homogeneous Moran sets. We refer \cite{FWW97} for the definition and the results about homogeneous Moran sets\footnote{For our use, we change the definition a little. It can be seen that the same results hold.}.
	
	Let $\{n_k\}_{k=1}^\infty$ be a sequence of positive integers and $\{c_k\}_{k=1}^\infty$ be a sequence of positive numbers satisfying that $n_kc_k\leq1\ (k=1,2,\cdots)$ and $c_k<c\ (k=1,2,\cdots)$ for some $0<c<1$. Let $D_0=\{\emptyset\}$, $D_k=\left\{(i_1,\cdots,i_k)\left|1\leq i_j\leq n_j,j=1,\dots,k\right.\right\}$
	for each $k=1,2,\cdots$ and $D=\bigcup_{k\geq 0}D_k$. If $\sigma=(\sigma_1,\dots,\sigma_k)\in D_k$ and $\tau=(\tau_1,\dots,\tau_m)\in D_m$, we write $\sigma*\tau=(\sigma_1,\dots,\sigma_k,\tau_1,\dots,\tau_m)\in D_{k+m}$.
	
	\begin{dfn}[Homogeneous Moran Sets]
		A collection $\mathscr{F}=\{J_\sigma\}_{\sigma\in D}$ of closed intervals of $\T$ has {\bf homogeneous Moran structure} about $\{n_k\}_{k=1}^\infty$ and $\{c_k\}_{k=1}^\infty$ if it satisfies the following:
		\begin{enumerate}
			\renewcommand{\labelenumi}{\rm{(\roman{enumi})}}
			\item $J_\emptyset=\T$,
			\item for each $k=0,1,\cdots$ and $\sigma\in D_k$, $J_{\sigma*i}\ (i=1,\dots,n_{k+1})$ are subintervals of $J_\sigma$ and $\mathring{J}_{\sigma*i}\ (i=1,\dots,n_{k+1})$ are pairwise disjoint (where $\mathring{A}$ denotes the interior of $A$ with respect to the usual topology of $\T$),
			\item for each $k=1,2,\cdots$, $\sigma\in D_{k-1}$ and $1\leq i\leq n_k$, we have
			$$
			c_k=\frac{|J_{\sigma*i}|}{|J_\sigma|}
			$$
			(where $|A|$ denotes the length of a interval $A$ of $\T$).
		\end{enumerate}
		If $\mathscr{F}$ is a collection of closed intervals having homogeneous Moran structure, we write $$E(\mathscr{F})=\bigcap_{k\geq 0}\bigcup_{\sigma\in D_k}J_\sigma$$
		and call $E(\mathscr{F})$ the {\bf homogeneous Moran set} determined by $\mathscr{F}$.
	\end{dfn}
	
	\begin{figure}[h]
		\centering
		\includegraphics[width=7cm]{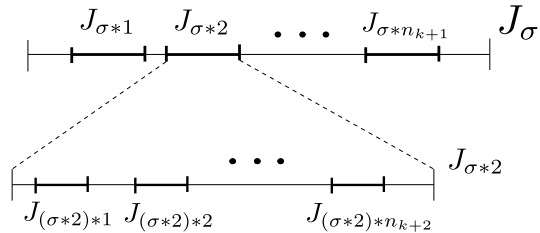}
		\caption{Homogeneous Moran structure}
	\end{figure}
	
	We write $\mathscr{M}\left(\{n_k\},\{c_k\}\right)$ for the set of homogeneous Moran sets determined by some collection $\mathscr{F}$ of closed intervals having homogeneous Moran structure about $\{n_k\}_{k=1}^\infty$ and $\{c_k\}_{k=1}^\infty$. Then we have the following estimate of Hausdorff dimension of homogeneous Moran sets.
	
	\begin{thm}[{\cite[Theorem 2.1]{FWW97}}]\label{dimMoranset}
		Let
		$$
		s_1=\liminf_{k\to\infty}\frac{\log n_1\cdots n_k}{-\log c_1\cdots c_k},\quad s_2=\liminf_{k\to\infty}\frac{\log n_1\cdots n_k}{-\log c_1\cdots c_kc_{k+1} n_{k+1}}.
		$$
		Then, for any $E\in\mathscr{M}\left(\{n_k\},\{c_k\}\right)$, we have
		$$
		s_2\leq \dim_H E\leq s_1.
		$$
	\end{thm}
	
	We  begin the proof of Theorem \ref{doesnotconverge}. We take arbitrary $0<r<1$ near $1$. It is sufficient to construct a subset $E$ of $J$ with Hausdorff dimension $\geq r$.
	
	We first construct divergent subsequences $\{N_k\}_{k=1}^\infty$ and $\{L_k\}_{k=1}^\infty$ in $\N$ we need by induction. 
	We take a countable subset $\{\psi_i\}_{i=1}^\infty\subset C(\T)$ so that $0<\psi_i\leq 1$ on $\T$ for each $i$ and, for a sequence $\{\mu_n\}_{n=1}^\infty\subset M(\T)$ and $\mu\in M(\T)$, $\mu_n\to\mu$ as $n\to\infty$ is equivalent to $\int_\T \psi_i\ d\mu_n\to\int_\T \psi_i\ d\mu$ as $n\to\infty$ for any $i$.
	For each $d\in\N$, we write $I_{d,j}=\left[j/d,(j+1)/d\right] \mod \Z$ for $j=0,\dots,d-1$ and $I_d=\left\{I_{d,j}\left|j=0,\dots,d-1\right.\right\}$. We remark that $I_{ab}$ is a common Markov partition of $\T$ with respect to $T_a,T_b$ and $T_{ab}$.
	We put $N_0=L_0=0$. Let $k>0$ and suppose that $N_i,L_i$ are determined for $i=0,\cdots,k-1$ so that $L_{i-1}<N_i<\lfloor rL_i\rfloor<L_i$ for $1\leq i<k$.
	For $N\in\N$, we define
	\begin{equation}\label{XkN}
		X_{k,N}=\left\{x\in\T\left| \left|\frac{1}{N^2}\sum_{m,n=0}^{N-1}\psi_i(T_a^mT_b^nx)-\int_\T \psi_i\ dm_\T\right|<\frac{1}{3k},1\leq i\leq k\right.\right\}.
	\end{equation}
	Then, by Birkhoff's ergodic theorem for $m_\T\in E_{\times a,\times b}(\T)$, we have
	\begin{equation*}\label{largeXkN}
		m_\T(X_{k,N})>r
	\end{equation*}
	for sufficiently large $N\in\N$. We take $l_k\in\N$ so that
	\begin{equation}\label{psiicont}
		\left|\psi_i(x)-\psi_i(y)\right|<\frac{1}{3k},\quad i=1,\dots,k
	\end{equation}
	for any $x,y\in\T$ such that $|x-y|\leq (ab)^{-l_k}$.
	We take $N_k\in\N$ such that $N_k>L_{k-1}+l_k$, $m_\T(X_{k,N_k})>r$,
	\begin{equation}\label{largeNk}
		\frac{N_k^2-(N_k-L_{k-1}-l_k)^2}{N_k^2}<\frac{1}{6k},
	\end{equation}
	and
	\begin{equation*}
		\frac{\sum_{i=0}^{k-1}(N_i+L_i)}{N_k}<\frac{1}{k}.
	\end{equation*}
	Let $x\in X_{k,N_k}$. For $y\in\T$, suppose that $T_{ab}^{L_{k-1}}x$ and $T_{ab}^{L_{k-1}}y$ are contained in the same element of $I_{(ab)^{N_k-L_{k-1}}}$. Then, for any $L_{k-1}\leq m,n<N_k-l_k$, $T_a^mT_b^nx$ and $T_a^mT_b^ny$ are contained in the same element of $I_{(ab)^{l_k}}$.
	From the definition of $X_{k,N_k}$ (\ref{XkN}) and the inequalities (\ref{psiicont}) and (\ref{largeNk}), we have, for $1\leq i\leq k$,
	\begin{align*}
		&\ \left|\frac{1}{N_k^2}\sum_{m,n=0}^{N_k-1}\psi_i(T_a^mT_b^ny)-\int_\T\psi_i\ dm_\T \right|\\
		\leq&\ \left|\frac{1}{N_k^2}\sum_{m,n=0}^{N_k-1}\psi_i(T_a^mT_b^ny)-\frac{1}{N_k^2}\sum_{m,n=0}^{N_k-1}\psi_i(T_a^mT_b^nx) \right|
		+\left|\frac{1}{N_k^2}\sum_{m,n=0}^{N_k-1}\psi_i(T_a^mT_b^nx)-\int_\T\psi_i\ dm_\T \right|\\
		\leq&\  \left|\frac{1}{N_k^2}\sum_{m,n=L_{k-1}}^{N_k-l_k-1}\psi_i(T_a^mT_b^ny)-\frac{1}{N_k^2}\sum_{m,n=L_{k-1}}^{N_k-l_k-1}\psi_i(T_a^mT_b^nx) \right|+2\cdot\frac{N_k^2-(N_k-L_{k-1}-l_k)^2}{N_k^2}+\frac{1}{3k}\\
		<&\ \frac{1}{k}.
	\end{align*}
	We take $L_k\in\N$ so that $\lfloor rL_k\rfloor>N_k$ and 
	$$
	\frac{\sum_{i=0}^{k-1}(N_i+L_i)+N_k}{L_k}<\frac{1}{k}.
	$$
	
	As a result, we obtain divergent subsequences $\{N_k\}_{k=1}^\infty$ and $\{L_k\}_{k=1}^\infty$ in $\N$ such that
	\begin{enumerate}
		\renewcommand{\labelenumi}{\rm{(\roman{enumi})}}
		\item$$L_{k-1}<N_k<\lfloor rL_k\rfloor<L_k,\quad k=1,2,\cdots,$$
		where we write $L_0=0$,
		\item\label{ll} for $k=1,2,\cdots$, 
		$m_\T(X_{k,N_k})>r$,
		\item for $k=1,2,\cdots$, if $x\in X_{k,N_k}$ and $y\in\T$ satisfies that $T_{ab}^{L_{k-1}}x$ and $T_{ab}^{L_{k-1}}y$ are contained in the same element of $I_{(ab)^{N_k-L_{k-1}}}$, we have
		$$
		\left|\frac{1}{N_k^2}\sum_{m,n=0}^{N_k-1}\psi_i(T_a^mT_b^ny)-\int_\T\psi_i\ dm_\T \right|<\frac{1}{k}
		$$
		for $1\leq i\leq k$,
		\item
		$$
		\lim_{k\to\infty}\frac{\sum_{i=1}^{k-1}(N_i+L_i)}{N_k}=0,\quad\lim_{k\to\infty}\frac{\sum_{i=1}^{k-1}(N_i+L_i)+N_k}{L_k}=0.
		$$
	\end{enumerate}
	
	Next, we construct a subset $E$ mentioned above. We write $\Omega=\{0,1,\dots,ab-1\}^{\Z_{\geq0}}$ and $\pi:\Omega\to\T$ for the coding map about the Markov partition $I_{ab}$ with respect to $T_{ab}$, that is, for $\omega=(\omega_0,\omega_1,\dots)\in\Omega$, $x=\pi(\omega)\in\T$ is the element such that $\{x\}=\bigcap_{i=0}^\infty T_{ab}^{-i}I_{ab,\omega_i}$.
	For $k=1,2,\dots$, we define
	\begin{equation*}
		\Lambda_k=\left\{\omega\in\Omega\left|\ \omega_i=\omega'_i,L_{k-1}\leq i<N_k\text{ for some }\omega'\in\pi^{-1}X_{k,N_k}\right.\right\}.
	\end{equation*}
	For $L\leq N\in\Z_{\geq 0}$, we call a subset $C\subset \Omega$ a cylinder set on $[L,N]$ if $C=C_{L,N}(\omega')=\left\{\omega\in\Omega\left|\omega_i=\omega'_i,L\leq i\leq N\right.\right\}$ for some $\omega'\in\Omega$. Then $\Lambda_k$ can be written as the finite and disjoint union of cylinder sets on $[L_{k-1},N_k-1]$:
	$$
	\Lambda_k=\bigsqcup_{C\in\mathscr{C}_k}C,
	$$
	where $\mathscr{C}_k=\left\{C_{L_{k-1},N_k-1}(\omega')\left|\omega'\in\pi^{-1}X_{k,N_k}\right.\right\}$. We have $\pi(\Lambda_k)=\bigcup_{C\in\mathscr{C}_k}\pi(C)\supset X_{k,N_k}$, $m_\T(\pi(C))=(ab)^{L_{k-1}-N_k}$ for each $C\in\mathscr{C}_k$ and $\pi(C)$ and $\pi(C')$ intersect only on $\Q/\Z\subset\T$ if $C,C'\in\mathscr{C}_k$ and $C\neq C'$. Hence, by the property (ii) of $\{N_k\}_{k=1}^\infty$, we have
	$$
	r<m_\T(X_{k,N_k})\leq m_\T(\pi(\Lambda_k))=\sum_{C\subset\mathscr{C}_k}m_\T\left(\pi(C)\right)=\left|\mathscr{C}_k\right|(ab)^{L_{k-1}-N_k}
	$$
	and
	\begin{equation}\label{Ck}
		\left|\mathscr{C}_k\right|>r(ab)^{N_k-L_{k-1}}.
	\end{equation}
	We define
	\begin{equation*}
		\Lambda=\left\{\omega\in\Omega\left|\ \omega\in\Lambda_k\text{ and }\omega_i=0,\lfloor rL_k\rfloor\leq i<L_k\text{ for any }k=1,2,\dots\right.\right\}
	\end{equation*}
	and
	\begin{equation*}
		E=\pi(\Lambda).
	\end{equation*}
	We show that this $E$ is a subset of $J$ such that $\dim_HE\geq r$.
	
	\begin{prop}\label{EsubsetJ}
		$$E\subset J.$$
	\end{prop}
	
	\begin{proof}
		Let $x\in E$ and take $\omega\in\Lambda$ such that $x=\pi(\omega)$. For each $k\geq 1$, since $\omega\in\Lambda_k$, we can take $\omega'\in \Omega$ such that $x'=\pi(\omega')\in X_{k,N_k}$ and $\omega_i=\omega'_i$ for $L_{k-1}\leq i< N_k$. Then it follows that $T_{ab}^{L_{k-1}}x'$ and $T_{ab}^{L_{k-1}}x$ are contained in the same element of $I_{(ab)^{N_k-L_{k-1}}}$ and, from the property (iii) of $\{N_k\}_{k=1}^\infty$, we have
		$$
		\left|\frac{1}{N_k^2}\sum_{m,n=0}^{N_k-1}\psi_i(T_a^mT_b^nx)-\int_\T\psi_i\ dm_\T \right|<\frac{1}{k}
		$$
		for $1\leq i\leq k$. Hence we have
		$$
		\frac{1}{N_k^2}\sum_{m,n=0}^{N_k-1}\psi_i(T_a^mT_b^nx)\xrightarrow[k\to\infty]{}\int_\T\psi_i\ dm_\T
		$$
		for any $i$. This fact implies that
		\begin{equation}\label{convergentpart}
			\delta_{\times a,\times b,x}^{N_k}\xrightarrow[k\to\infty]{}m_\T.
		\end{equation}
		
		Next, we show that $\delta_{\times a,\times b,x}^{L_k}$ does not converge to $m_\T$ as $k\to\infty$.
		We take $l\in\N$ such that $(ab)^{-l}<2^{-1}(1-r)^2$ and $\varphi\in C(\T)$ such that $0\leq \varphi\leq 1$ on $\T$, $\varphi=1$ on $[0,(ab)^{-l}] \mod\Z$ and $(ab)^{-l}\leq \int_\T \varphi\  dm_\T<2^{-1}(1-r)^2$. For sufficiently large $k$, we have
		$$
		\lfloor rL_k\rfloor\leq rL_k<L_k-l,\quad
		\frac{2(1-r)lL_k-l^2}{L_k^2}<\frac{1}{2}(1-r)^2.
		$$
		Furthermore, since $\omega\in\Lambda$, it follows that $T_{ab}^ix\in [0,(ab)^{-1}] \mod\Z$ for any $\lfloor rL_k\rfloor\leq i<L_k$ and, hence $T_a^mT_b^nx\in [0,(ab)^{-l}]\mod\Z$ for any $\lfloor rL_k\rfloor\leq m,n<L_k-l$. Then we have
		\begin{align*}
			\frac{1}{L_k^2}\sum_{m,n=0}^{L_k-1}\varphi(T_a^mT_b^nx)\geq&\ \frac{1}{L_k^2}\sum_{m,n=\lfloor rL_k\rfloor}^{L_k-l-1}\varphi(T_a^mT_b^nx)\\
			=&\ \frac{(L_k-l-\lfloor rL_k\rfloor)^2}{L_k^2}\\
			\geq&\ \frac{\left((1-r)L_k-l\right)^2}{L_k^2}\\
			=&\ (1-r)^2-\frac{2(1-r)lL_k-l^2}{L_k^2}\\
			>&\ \frac{1}{2}(1-r)^2.
		\end{align*}
		Hence we have
		\begin{align*}
			\liminf_{k\to\infty}\frac{1}{L_k^2}\sum_{m,n=0}^{L_k-1}\varphi(T_a^mT_b^nx)\geq&\ \frac{1}{2}(1-r)^2\\
			>&\int_\T \varphi\ dm_\T.
		\end{align*}
		This implies that $\delta_{\times a,\times b,x}^{L_k}$ does not converge to $m_\T$ as $k\to\infty$. This and (\ref{convergentpart}) imply that $x\in J$ and we complete the proof.
	\end{proof}
	
	\begin{prop}\label{dimHE}
		$$\dim_H E\geq r.$$
	\end{prop}
	
	\begin{proof}
		We show that $E$ is a homogeneous Moran set and use Theorem \ref{dimMoranset}.
		Let $k=1,2,\cdots$. First, we notice that, for $\omega\in\Lambda$, $\omega\in C$ for some $C\in\mathscr{C}_k$: the subfamily of cylinder sets on $[L_{k-1},N_k-1]$. We define
		\begin{equation*}\label{MoranE1}
			n_{k,1}=\left|\mathscr{C}_k\right|,\quad c_{k,1}=(ab)^{-(N_k-L_{k-1})}.
		\end{equation*}
		Second, we notice that, for $\omega\in\Lambda$, $\omega_i$ is arbitrary for $N_k\leq i<\lfloor rL_k\rfloor$.
		For each $N_k\leq i<\lfloor rL_k\rfloor$, we define
		\begin{equation*}\label{MoranE2}
			n_{k,2,i}=ab,\quad c_{k,2,i}=(ab)^{-1}.
		\end{equation*}
		And finally, we notice that, for $\omega\in\Lambda$, $\omega_i=0$ for $\lfloor rL_k\rfloor\leq i<L_k$. We define
		\begin{equation*}\label{MoranE3}
			n_{k,3}=1,\quad c_{k,3}=(ab)^{-(L_k-\lfloor rL_k\rfloor)}.
		\end{equation*}
		We write 
		\begin{align*}
			&\{n_l\}_{l=1}^\infty=\{n_{1,1},\dots,n_{k-1,3},n_{k,1},n_{k,2,N_k},\dots,n_{k,2,\lfloor rL_k\rfloor-1},n_{k,3},n_{k+1,1},\dots\},\\
			&\{c_l\}_{l=1}^\infty=\{c_{1,1},\dots,c_{k-1,3},c_{k,1},c_{k,2,N_k},\dots,c_{k,2,\lfloor rL_k\rfloor-1},c_{k,3},c_{k+1,1},\dots\}.
		\end{align*}
		Then, by the definition of $E$, it is seen that $E\in\mathscr{M}(\{n_l\},\{c_l\})$. Hence, by Theorem \ref{dimMoranset}, we have
		\begin{equation}\label{dimHEs2}
			\dim_H E\geq s_2=\liminf_{l\to\infty}\frac{\log n_1\cdots n_l}{-\log c_1\cdots c_lc_{l+1}n_{l+1}}.
		\end{equation}
		We estimate the right hand side of (\ref{dimHEs2}).
		
		Suppose $n_l=n_{k,1}$ and $c_l=c_{k,1}$. Then $n_{l+1}=n_{k,2,N_k}=ab$ and $c_{l+1}=c_{k,2,N_k}=(ab)^{-1}$.
		From the inequality (\ref{Ck}), it follows that
		\begin{align*}
			n_1\cdots n_l=&\prod_{j=1}^{k-1}(n_{j,1}n_{j,2,N_j}\cdots n_{j,2,\lfloor rL_j\rfloor-1}n_{j,3})\cdot n_{k,1}\\
			=&\prod_{j=1}^{k-1}\left(|\mathscr{C}_j|(ab)^{\lfloor rL_j\rfloor-N_j}\right)\cdot |\mathscr{C}_k|\\
			>&\prod_{j=1}^{k-1}(r(ab)^{\lfloor rL_j\rfloor-L_{j-1}})\cdot r(ab)^{N_k-L_{k-1}},
		\end{align*}
		and
		\begin{align*}
			c_1\cdots c_l=&\prod_{j=1}^{k-1}(c_{j,1}c_{j,2,N_j}\cdots c_{j,2,\lfloor rL_j\rfloor-1}c_{j,3})\cdot c_{k,1}\\
			=&\prod_{j=1}^{k-1}((ab)^{-(L_j-L_{j-1})})\cdot (ab)^{-(N_k-L_{k-1})}\\
			=&\ (ab)^{-N_k}.
		\end{align*}
		Hence we have
		\begin{align}\label{leqk1}
			\frac{\log n_1\cdots n_l}{-\log c_1\cdots c_{l}c_{l+1}n_{l+1}}\geq&\  \frac{\log\left\{\prod_{j=1}^{k-1}(r(ab)^{\lfloor rL_j\rfloor-L_{j-1}})\cdot r(ab)^{N_k-L_{k-1}}\right\}}{-\log\left\{(ab)^{-N_k}\right\}}\nonumber\\
			=&\ \frac{\sum_{j=1}^{k-1}\left\{\log r+(\lfloor rL_j\rfloor-L_{j-1})\log(ab)\right\}+\log r+(N_k-L_{k-1})\log(ab)}{N_k\log(ab)}\nonumber\\
			=&\ \frac{k\log r}{N_k\log (ab)}+\frac{\sum_{j=1}^{k-1}(\lfloor rL_j\rfloor-L_{j-1})}{N_k}+\frac{N_k-L_{k-1}}{N_k}.
		\end{align}
		From the property (iv) of $\{N_k\}_{k=1}^\infty$ and $\{L_k\}_{k=1}^\infty$, the right hand side converges to $1$ as $k\to\infty$.
		
		Suppose $n_l=n_{k,2,i}$ and $c_l=c_{k,2,i}$ for some $N_k\leq i<\lfloor rL_k\rfloor$. Then
		$$
		n_{l+1}=
		\begin{cases}
			n_{k,2,i+1}=ab,&i\neq \lfloor rL_k\rfloor-1\\
			n_{3,k}=1,&i=\lfloor rL_k\rfloor-1
		\end{cases},\quad
		c_{l+1}=
		\begin{cases}
			c_{k,2,i+1}=(ab)^{-1},&i\neq \lfloor rL_k\rfloor-1\\
			c_{3,k}=(ab)^{-(L_k-\lfloor rL_k\rfloor)},&i=\lfloor rL_k\rfloor-1
		\end{cases}.
		$$
		From the inequality (\ref{Ck}), it follows that
		\begin{align*}
			n_1\cdots n_l=&\prod_{j=1}^{k-1}(n_{j,1}n_{j,2,N_j}\cdots n_{j,2,\lfloor rL_j\rfloor-1}n_{j,3})\cdot n_{k,1}n_{k,2,N_k}\cdots n_{k,2,i}\\
			>&\prod_{j=1}^{k-1}(r(ab)^{\lfloor rL_j\rfloor-L_{j-1}})\cdot r(ab)^{N_k-L_{k-1}}\cdot (ab)^{i-N_k+1}\\
			=&\prod_{j=1}^{k-1}(r(ab)^{\lfloor rL_j\rfloor-L_{j-1}})\cdot r(ab)^{i-L_{k-1}+1},
		\end{align*}
		and
		\begin{align*}
			c_1\cdots c_l=&\prod_{j=1}^{k-1}(c_{j,1}c_{j,2,N_j}\cdots c_{j,2,\lfloor rL_j\rfloor-1}c_{j,3})\cdot c_{k,1}c_{k,2,N_k}\cdots c_{k,2,i}\\
			=&\ (ab)^{-N_k}\cdot (ab)^{-(i-N_k+1)}\\
			=&\ (ab)^{-(i+1)}.
		\end{align*}
		Hence we have
		\begin{align}\label{leqk2}
			&\ \frac{\log n_1\cdots n_l}{-\log c_1\cdots c_{l}c_{l+1}n_{l+1}}\nonumber\\
			\geq&\ \frac{\log\left\{\prod_{j=1}^{k-1}(r(ab)^{\lfloor rL_j\rfloor-L_{j-1}})\cdot r(ab)^{i-L_{k-1}+1}\right\}}{-\log\left\{(ab)^{-(i+1)}\cdot c_{l+1}n_{l+1}\right\}}\nonumber\\
			=&\ \frac{\sum_{j=1}^{k-1}\left\{\log r+(\lfloor rL_j\rfloor-L_{j-1})\log(ab)\right\}+\log r+(i-L_{k-1}+1)\log(ab)}{(i+1)\log(ab)-\log(c_{l+1}n_{l+1})}.
		\end{align}
		If $i<\lfloor rL_k\rfloor-1$, then $c_{l+1}n_{l+1}=1$ and the right hand side of (\ref{leqk2}) is
		\begin{align}\label{leqk21}
			&\ \frac{\sum_{j=1}^{k-1}\left\{\log r+(\lfloor rL_j\rfloor-L_{j-1})\log(ab)\right\}+\log r+(i-L_{k-1}+1)\log(ab)}{(i+1)\log(ab)}\nonumber\\
			\geq&\ \frac{k\log r}{N_k\log(ab)}+\frac{\sum_{j=1}^{k-1}(\lfloor rL_j\rfloor-L_{j-1})}{\lfloor rL_k\rfloor}-\frac{L_{k-1}}{N_k}+1.
		\end{align}
		From the property (iv) of $\{N_k\}_{k=1}^\infty$ and $\{L_k\}_{k=1}^\infty$, the right hand side converges to $1$ as $k\to\infty$. If $i=\lfloor rL_k\rfloor-1$, then $c_{l+1}n_{l+1}=(ab)^{-(L_k-\lfloor rL_k\rfloor)}$ and the right hand side of (\ref{leqk2}) is
		\begin{align}\label{leqk22}
			&\ \frac{\sum_{j=1}^{k-1}\left\{\log r+(\lfloor rL_j\rfloor-L_{j-1})\log(ab)\right\}+\log r+(\lfloor rL_k\rfloor-L_{k-1})\log(ab)}{L_k\log(ab)}\nonumber\\
			\geq&\ \frac{k\log r}{L_k\log(ab)}+\frac{\sum_{j=1}^{k-1}(\lfloor rL_j\rfloor-L_{j-1})}{L_k}+\frac{\lfloor rL_k\rfloor-L_{k-1}}{L_k}.
		\end{align}
		From the property (iv) of $\{N_k\}_{k=1}^\infty$ and $\{L_k\}_{k=1}^\infty$, the left hand side converges to $r$ as $k\to\infty$.
		
		Suppose $n_l=n_{k,3}$ and $c_l=c_{k,3}$. Then $n_{l+1}=n_{k+1,1}=|\mathscr{C}_{k+1}|$, $c_{l+1}=c_{k+1,1}=(ab)^{-(N_{k+1}-L_k)}$ and, from the inequality (\ref{Ck}),
		$$
		c_{l+1}n_{l+1}>r.
		$$
		From (\ref{Ck}) again, it follows that
		\begin{align*}
			n_1\cdots n_l=&\prod_{j=1}^k(n_{j,1}n_{j,2,N_j}\cdots n_{j,2,\lfloor rL_j\rfloor-1}n_{j,3})\\
			>&\prod_{j=1}^k(r(ab)^{\lfloor rL_j\rfloor-L_{j-1}}),
		\end{align*}
		and
		\begin{align*}
			c_1\cdots c_l=&\prod_{j=1}^k(c_{j,1}c_{j,2,N_j}\cdots c_{j,2,\lfloor rL_j\rfloor-1}c_{j,3})\\
			=&\ (ab)^{-L_k}.
		\end{align*}
		Hence we have
		\begin{align}\label{leqk3}
			&\ \frac{\log n_1\cdots n_l}{-\log c_1\cdots c_{l}c_{l+1}n_{l+1}}\nonumber\\
			\geq&\ \frac{\log\left\{\prod_{j=1}^k(r(ab)^{\lfloor rL_j\rfloor-L_{j-1}})\right\}}{-\log\left\{(ab)^{-L_k}\cdot c_{l+1}n_{l+1}\right\}}\nonumber\\
			\geq&\ \frac{\sum_{j=1}^k\left\{\log r+(\lfloor rL_j\rfloor-L_{j-1})\log(ab)\right\}}{L_k\log(ab)-\log r}\nonumber\\
			=&\ \frac{k\log r}{L_k\log (ab)-\log r}+\frac{\sum_{j=1}^{k-1}(\lfloor rL_j\rfloor-L_{j-1})\log (ab)}{L_k\log(ab)-\log r}
			+\frac{(\lfloor rL_k\rfloor-L_{k-1})\log(ab)}{L_k\log(ab)-\log r}.
		\end{align}
		From the property (iv) of $\{N_k\}_{k=1}^\infty$ and $\{L_k\}_{k=1}^\infty$, the left hand side converges to $r$ as $k\to\infty$.
		
		From the inequalies (\ref{leqk1}), (\ref{leqk2}), (\ref{leqk21}), (\ref{leqk22}) and (\ref{leqk3}), we have
		$$
		s_2=\liminf_{l\to\infty}\frac{\log n_1\cdots n_l}{-\log c_1\cdots c_{l}c_{l+1}n_{l+1}}\geq r
		$$
		and, from the inequality (\ref{dimHEs2}), we complete the proof.
	\end{proof}
	
	By Proposition \ref{EsubsetJ} and \ref{dimHE}, we have $1\geq\dim_HJ\geq\dim_HE\geq r$ and $0<r<1$ is arbitrary. Hence we have $\dim_H J=1$ and complete the proof of Theorem \ref{doesnotconverge}.
	
	\section{Proof of Theorem \ref{entropy0} and \ref{distribution}}\label{proofentropy}
	In this section, we prove Theorem \ref{entropy0} and \ref{distribution}. First, we prove Theorem \ref{distribution} as the proof is more elementary than that of Theorem \ref{entropy0}.
	
	\begin{proof}[Proof of Theorem \ref{distribution}]
		Suppose $a$ and $b$ are multiplicatively independent. Let $0<t<\min\{\log b,(\log a)^2/\log b\}$ and $x\in\T\setminus K_t$. Assume there exists $f\in C(\T)$ such that $f\geq 0$ on $\T$ and
		$$\liminf_{N\to\infty}\frac{1}{N^2}\sum_{m,n=0}^{N-1}f(a^mb^nx)< \frac{t}{\log a}\int_\T f\ dm_\T.$$
		We can take $0<\varepsilon<1$ and some divergent subsequence $\{N_k\}_{k=1}^\infty$ in $\N$ such that
		\begin{equation}\label{contradist1}
			\frac{1}{N_k^2}\sum_{m,n=0}^{N_k-1}f(a^mb^nx)< \frac{t}{\log a}\int_\T f\ dm_\T-\varepsilon.
		\end{equation}
		for each $k$. Furthermore, since $M(\T)$ is compact with respect to the weak* topology, we can take $\{N_k\}_{k=1}^\infty$ so that $\delta_{\times a,\times b,x}^{N_k}$ converges to some $\mu\in M(\T)$ as $k\to\infty$. Then $\mu\in M_{\times a,\times b}(\T)$ and $\mu$ is an accumulation point of $\delta_{\times a,\times b,x}^N\ (N\in\N)$. Since $x\in\T\setminus K_t$, we have $h_\mu(T_a)>t$. Here we decompose $\mu$ into $\times a,\times b$ ergodic components. There exists a Borel probability measure $\tau$ on the compact and metrizable space $M_{\times a,\times b}(\T)$ such that $\tau(E_{\times a,\times b}(\T))=1$ and
		$$\int_\T \varphi\ d\mu=\int_{E_{\times a,\times b}(\T)}\int_\T \varphi\ d\nu d\tau(\nu)
		$$
		for any $\varphi\in C(\T)$. By the upper semicontinuity of $h_\nu(T_a)$, it can be seen that
		$$h_\mu(T_a)=\int_{E_{\times a,\times b}(\T)}h_\nu(T_a)d\tau(\nu)
		$$
		and, by Theorem \ref{RudolphJohnson}, $h_\nu(T_a)=0$ for any $\nu\in E_{\times a,\times b}(\T)\setminus\{m_\T\}$. Hence we have
		\begin{equation}\label{contradist2}
			t<h_\mu(T_a)=\tau(\{m_\T\})h_{m_\T}(T_a)=\tau(\{m_\T\})\log a.
		\end{equation}
		Letting $k\to\infty$ in the inequality (\ref{contradist1}), it follows from (\ref{contradist2}) that
		\begin{align*}
			\frac{t}{\log a}\int_\T f\ dm_\T-\varepsilon\geq&\int_\T f\ d\mu\\
			=&\int_{E_{\times a,\times b}(\T)}\int_\T f \ d\nu d\tau(\nu)\\
			\geq&\ \tau(\{m_\T\})\int_\T f\ dm_\T\\
			\geq&\ \frac{t}{\log a}\int_\T f\ dm_\T
		\end{align*}
		and this is a contradiction. Hence we have
		$$\liminf_{N\to\infty}\frac{1}{N^2}\sum_{m,n=0}^{N-1}f(a^mb^nx)\geq \frac{t}{\log a}\int_\T f\ dm_\T$$
		for any $f\in C(\T)$ such that $f\geq 0$ on $\T$.
		
		Let $U\subset\T$ be an open subset. For any $0<\varepsilon<1$, there exists $f\in C(\T)$ such that $0\leq f\leq 1$ on $\T$, $f=0$ on $\T\setminus U$ and $\int_\T f\ dm_\T \geq m_\T(U)-\varepsilon$. Then, by the statement above, it follows that
		\begin{align*}
			\liminf_{N\to\infty}\frac{1}{N^2}\left|\left\{(m,n)\in\Z^2\left|\ 0\leq m,n<N, a^mb^nx\in U\right.\right\}\right|\geq&\ \liminf_{N\to\infty}\frac{1}{N^2}\sum_{m,n=0}^{N-1}f(a^mb^nx)\\
			\geq&\ \frac{t}{\log a}\int_\T f\ dm_\T\\
			\geq&\ \frac{t}{\log a}\left(m_\T(U)-\varepsilon\right).
		\end{align*}
		By letting $\varepsilon\to 0$, we have
		$$\liminf_{N\to\infty}\frac{1}{N^2}\left|\left\{(m,n)\in\Z^2\left|\ 0\leq m,n<N, a^mb^nx\in U\right.\right\}\right|\geq\frac{t}{\log a}\cdot m_\T (U)
		$$
		and complete the proof.
	\end{proof}
	
	Next, we prove Theorem \ref{entropy0}.
	The following argument can be thought as an extension of that in \cite{Bo} to the $\Z_{\geq0}^2$-action by $T_a$ and $T_b$.
	Let $k\in\N$. $p=(p_1,\dots,p_k)\in\R^k$ is a {\bf $\boldsymbol{k}$-distribution} if $\sum_{i=1}^kp_i=1$ and $p_i\geq0$. For such a $p$, we write $H(p)=-\sum_{i=1}^kp_i\log p_i$ for the entropy of $p$. If $N\in\N$ and $c=(c_1,\dots,c_N)\in\{1,\dots,k\}^N$, we define the $k$-distribution $\dist(c)=(p_1,\dots,p_k)$, where $p_i=N^{-1}\left|\left\{n\in\{1,\dots,N\}\left|\ c_n=i\right.\right\}\right|$.
	
	\begin{lem}\label{fund}
		For $k,N\in\N$ and $t>0$, let
		$$R(k,N,t)=\left\{c\in\{1,\dots,k\}^N\left|H(\dist(c))\leq t\right.\right\}.$$
		Then, fixing $k$ and $t$, 
		$$\limsup_{N\to\infty}\frac{1}{N}\log\left|R(k,N,t)\right|\leq t.$$
	\end{lem}
	\begin{proof}
		See \cite[Lemma 4]{Bo}.
	\end{proof}
	Suppose $\beta=\{\beta_1,\dots,\beta_k\}$ is a finite cover of $\T$. For $x\in\T$ and $N\in\N$, we say that $(\beta_{i_0},\dots,\beta_{i_{N-1}})\in\beta^N$ is an {\bf $\boldsymbol{N}$-choice for $\boldsymbol{x}$ with respect to $\boldsymbol{T_a}$ and $\boldsymbol{\beta}$} if $T_a^nx\in\beta_{i_n}$ for $0\leq n< N$.  Then $(\beta_{i_0},\dots,\beta_{i_{N-1}})$ gives a $k$-distribution $q(\beta_{i_0},\dots,\beta_{i_{N-1}})=\dist(i_0,\dots,i_{N-1})$. We write $\Dist_\beta(x,N)$ for the set of such $k$-distributions obtained for all $N$-choices for $x$.
	
	Suppose $B=\{B_i\}$ is a finite cover of $\T$. For $E\subset\T$, we write $E\prec B$ if $E\subset B_i$ for some $B_i\in B$ and, for a family of subsets $E=\{E_j\}$, $E\prec B$ if $E_j\prec B$ for any $E_j\in E$. For a map $T:\T\to\T$, $l\in\N$ and a family of subsets $E=\{E_j\}$, we define $T^{-l}E=\{T^{-l}E_j\}$.
	
	\begin{lem}\label{keylemma}
		Let $B=\{B_i\}$ be a finite open cover of $\T$ such that every $B_i\in B$ is an open interval on $\T$ such that $|B_i|<1/(1+a)$ and, for each $M\in\N$, $\beta_M$ be a finite cover of $\T$ such that $\beta_M\prec T_a^{- l}B$ for $0\leq l<M$.
		For $0<t<\log a$, we define $Q\left(t,\left\{\beta_M\right\}_{M\in\N}\right)$ as the set of $x\in\T$ satisfying the following: for any $0<\varepsilon<1$ and $M_0\in\N$, there exists $M\geq M_0$ such that, 
		$$
		\text{for infinitely many } N\in\N,\  \frac{1}{M}H(q)\leq t+\varepsilon \text{ for some } q\in\bigcup_{0\leq n<tN/\log b}\Dist_{\beta_M}(T_b^nx,N).$$
		Then we have
		$$\dim_H Q\left(t,\left\{\beta_M\right\}_{M\in\N}\right)\leq\frac{2t}{\log a+t}.$$
	\end{lem}
	\begin{proof}
		For each $M\in\N$, let $\beta_M=\left\{\beta_{M,1},\dots,\beta_{M,k_M}\right\},k_M=\left|\beta_M\right|$.
		We take $0<\varepsilon<3^{-1}(\log a-t)$. By Lemma \ref{fund}, there exists $N_{\varepsilon,M}\in\N$ such that
		\begin{equation}\label{fundapp}
			\left|R(k_M,N,M(t+\varepsilon))\right|\leq e^{NM(t+2\varepsilon)}
		\end{equation}
		for any $N\geq N_{\varepsilon,M}$.
		We take $M_0\in\N$ such that $M_0\geq t^{-1}\log b$. Since $H(p)$ is uniformly continuous in a $k_M$-distribution $p$, we can see that, for any $x\in Q\left(t,\left\{\beta_M\right\}_{M\in\N}\right)$, there exists $M\geq M_0$ such that,
		$$
		\text{for infinitely many } N\in\N,\  \frac{1}{M}H(q)\leq t+\varepsilon \text{ for some } q\in\bigcup_{0\leq n<tMN/\log b}\Dist_{\beta_M}(T_b^nx,MN).$$
		Indeed, we obtain this by adding some $0\leq l<M$ to $N$ in the definition of $Q\left(t,\left\{\beta_M\right\}_{M\in\N}\right)$ for $\varepsilon/2$. For each $M\in\N$, we take $N'_{\varepsilon,M}\in\N$ such that
		\begin{equation}\label{NM}
			N'_{\varepsilon,M}\geq N_{\varepsilon,M}, \quad M^2(k_M)^M\sum_{N\geq N'_{\varepsilon,M}}Ne^{-\varepsilon MN}<\frac{1}{2^M}.
		\end{equation}
		
		For each $M,N\in\N$ and $x\in\T$, we take $MN$-choice $(\beta_{M,i_0(x)},\dots,\beta_{M,i_{MN-1}(x)})$ for $x$ with respect to $T_a$ and $\beta_M$ such that
		\begin{equation}\label{min_choice}
			H(q(\beta_{M,i_0(x)},\dots,\beta_{M,i_{MN-1}(x)}))=\min_{q\in\Dist_{\beta_M}(x,MN)}H(q). 
		\end{equation}
		For $0\leq l<M$, we define a $k_M$-distribution
		$$q_{M,l}(x,N)=\dist(i_l(x),i_{M+l}(x),\dots,i_{M(N-1)+l}).$$
		Then $q(\beta_{M,i_0(x)},\dots,\beta_{M,i_{MN-1}(x)})=M^{-1}\sum_{l=0}^{M-1}q_{M,l}(x,N)$. Hence, by the concavity of $H(p)$ in a $k_M$-distribution $p$, we have $H(q_{M,l}(x,N))\leq H(q(\underline{\beta_M}_{MN}(x)))$ for some $0\leq l<M$, depending on $M,N$ and $x$.
		
		For $M\geq M_0$, $N\geq N'_{\varepsilon,M}$ and $n\in\Z$ with $0\leq n<tMN/\log b$ and $0\leq l<M$, we define
		$$S(M,N,n,l)=\left\{x\in\T\left|H(q_{M,l}(T_b^nx,N))\leq M(t+\varepsilon)\right.\right\}.$$
		Then we have
		$$Q\left(t,\left\{\beta_M\right\}_{M\in\N}\right)\subset\bigcup_{\substack{M\geq M_0,N\geq N'_{\varepsilon, M}\\0\leq n<tMN/\log b,0\leq l<M}}S(M,N,n,l).$$
		Let $M\geq M_0,N\geq N'_{\varepsilon,M},0\leq n<tMN/\log b,0\leq l<M$ and $x\in S(M,N,n,l)$.
		For the $MN$-choice
		$(\beta_{M,i_0(T_b^nx)},\dots,\beta_{M,i_{MN-1}(T_b^nx)})$ for $T_b^nx$ with respect to $T_a$ and $\beta_M$ as (\ref{min_choice}), we have
		$$(i_l(T_b^nx),i_{M+l}(T_b^nx),\dots,i_{M(N-1)+l}(T_b^nx))\in R(k_M,N,M(t+\varepsilon)).$$ 
		We define
		\begin{align*}
			A_{M,l}(T_b^nx,N)=&\left\{y\in\T\left|T_a^jy\in\beta_{M,i_j(T_b^nx)}\text{\rm{ for }}0\leq j<l,\right.\right.\\
			&\qquad \qquad \qquad \left.T_a^{Mr+l}y\in\beta_{M,i_{Mr+l}(T_b^n(x))}\text{\rm{ for }}0\leq r<N\right\}.
		\end{align*}
		Then, by the assumption of  $B$ and $\beta_M$, $A_{M,l}(T_b^nx,N)\prec T_a^{-j}B$ for $0\leq j<MN$. Hence, by the assumption that $|B_i|<1/(a+1)$ for each $B_i\in B$, we have $\diam{A_{M,l}(T_b^nx,N)}<a^{-MN+1}$, where $\diam A$ denotes the diameter $\sup_{x,y\in A}|x-y|$ of $A\subset\T$ with respect to the standard metric of $\T$. We devide $A_{M,l}(T_b^nx,N)$ as $A_{M,l}(T_b^nx,N)=\bigsqcup_{s=0}^{b-1}A_{M,l}^s(T_b^nx,N)$, where $A_{M,l}^s(T_b^nx,N)=A_{M,l}(T_b^nx,N)\cap(\left[s/b,(s+1)/b\right) \mod \Z)$, then $x\in T_b^{-n}A_{M,l}(T_b^nx,N)=\bigsqcup_{s=0}^{b-1}T_b^{-n}A_{M,l}^s(T_b^nx,N)$. For each $s=0,\dots,b-1$, we get the $b^n$ components of
		$T_b^{-n}A_{M,l}^s(T_b^nx,N)$, which we write $E_{M,l}^{s,u}(x,N,n),u=1,\dots,b^n$, satisfying 
		\begin{equation}\label{diamEM0}
			\diam E_{M,l}^{s,u}(x,N,n)<b^{-n}a^{-MN+1}.
		\end{equation}
		
		We define
		\begin{align*}
			E(M_0)=&\left\{E_{M,l}^{s,u}(x,N,n)\left| M\geq M_0,N\geq N'_{\varepsilon,M},0\leq n<tMN/\log b,0\leq l<M, \right.\right.\\
			&\qquad\qquad\qquad\qquad\qquad\qquad x\in S(M,N,n,l), s=0,\dots,b-1,u=1,\dots,b^n\bigr\},
		\end{align*}
		then $E(M_0)$ is a cover of $Q\left(t,\left\{\beta_M\right\}_{M\in\N}\right)$ such that $\diam E(M_0)\leq a^{-M_0+1}$. 
		Fix $M\geq M_0,N\geq N'_{\varepsilon,M},0\leq n<tMN/\log b$ and $0\leq l<M$. The number of $A_{M,l}(T_b^nx,N)\ (x\in S(M,N,n,l))$ is bounded by $|\beta_M|^l\left|R(k_M,N,M(t+\varepsilon))\right|=(k_M)^l\left|R(k_M,N,M(t+\varepsilon))\right|$. Hence the number of $E_{M,l}^{s,u}(x,N,n)\ (x\in S(M,N,n,l),s=0,\dots,b-1,u=1,\dots,b^n)$ is bounded by $b^{n+1}(k_M)^l\left|R(k_M,N,M(t+\varepsilon))\right|$.
		We put $\lambda=(\log a-t-3\varepsilon)/(\log a+t)$. Since $\log a-t>3\varepsilon$, we have $\lambda>0$. We also have $1-\lambda=((1+\lambda)t+3\varepsilon)/\log a>0$.
		Using the inequalities (\ref{fundapp}) and (\ref{diamEM0}), we have
		\begin{align*}
			&\sum_{E\in E(M_0)}(\diam E)^{1-\lambda}\\
			\leq&\sum_{\substack{M\geq M_0,N\geq N'_{\varepsilon,M},\\0\leq n<tMN/\log b,0\leq l<M}}b^{n+1}(k_M)^l\left|R(k_M,N,M(t+\varepsilon))\right|\left(b^{-n}a^{-MN+1}\right)^{1-\lambda}\\
			\leq&\ b\sum_{M\geq M_0}M(k_M)^M\sum_{N\geq N'_{\varepsilon,M},0\leq n<tM N/\log b}b^{\lambda n}\left|R(k_M,N,M(t+\varepsilon))\right|
			a^{(-MN+1)((1+\lambda)t+3\varepsilon)/\log a}\\
			\leq&\ b\sum_{M\geq M_0}M(k_M)^M\sum_{N\geq N'_{\varepsilon,M},0\leq n<tM N/\log b}b^{\lambda n}e^{MN(t+2\varepsilon)}e^{(-MN+1)((1+\lambda)t+3\varepsilon)}\\
			=&\ e^{(1+\lambda)t+3\varepsilon}b\sum_{M\geq M_0}M(k_M)^M\sum_{N\geq N'_{\varepsilon,M},0\leq n<tM N/\log b}b^{\lambda n}e^{-MN(\lambda t+\varepsilon)}\\
			=&\ e^{(1+\lambda)t+3\varepsilon}b\sum_{M\geq M_0}M(k_M)^M\sum_{N\geq N'_{\varepsilon,M}}\left(\frac{tMN}{\log b}+1\right)b^{tMN\lambda/\log b}e^{-MN(\lambda t+\varepsilon)}\\
			\leq&\ \frac{2te^{(1+\lambda)t+3\varepsilon}b}{\log b}\sum_{M\geq M_0}M^2(k_M)^M\sum_{N\geq N'_{\varepsilon,M}}Ne^{-MN\varepsilon}\\
			\leq&\ \frac{2te^{(1+\lambda)t+3\varepsilon}b}{\log b}\sum_{M\geq M_0}\frac{1}{2^M}.
		\end{align*}
		The last inequality is due to (\ref{NM}).
		When $M_0\to\infty$, the right hand side converges to $0$. This implies that
		$\dim_H Q\left(t,\left\{\beta_M\right\}_{M\in\N}\right)\leq 1-\lambda=(2t+3\varepsilon)/(\log a+t)$. 
		By $\varepsilon\to 0$, we obtain the lemma.
	\end{proof}
	
	Before starting a proof of Theorem \ref{entropy0}, we prepare a notion. Suppose $\beta=\left\{\beta_1,\dots,\beta_k\right\}$ is a finite cover of $\T$. For $x\in\T$ and $N\in\N$, we say that $(\beta_{i_{m,n}})_{0\leq m,n<N}\in\beta^{\left\{(m,n)\left|0\leq m,n<N\right.\right\}}$ is an {\bf $\boldsymbol{N}$-choice for $\boldsymbol{x}$ with respect to $\boldsymbol{T_a,T_b}$ and $\boldsymbol{\beta}$} if $T_a^mT_b^nx\in\beta_{i_{m,n}}$ for $0\leq m,n<N$. Then $(\beta_{i_{m,n}})_{0\leq m,n<N}$ gives a $k$-distribution $q((\beta_{i_{m,n}})_{0\leq m,n<N})=\dist((i_{m,n})_{0\leq m,n<N})$.
	We notice that, if $\underline{\beta}=(\beta_{i_{m,n}})_{0\leq m,n<N}$ is an $N$-chioce for $x$ with respect to $T_a,T_b$ and $\beta$, then, for $0\leq n<N$, $\underline{\beta}_n=\left(\beta_{i_{0,n}},\dots,\beta_{i_{N-1,n}}\right)$ is an $N$-choice for $T_b^nx$ with respect to $\beta$ and $T_a$, and $q(\underline{\beta})=N^{-1}\sum_{n=0}^{N-1}q(\underline{\beta}_n)$.
	
	\begin{proof}[Proof of Theorem \ref{entropy0}]
		Let $B$ be a finite open cover of $\T$ as in Lemma \ref{keylemma} and $\alpha$ be a finite Borel partition of $\T$ such that $\overline{\alpha_i}\prec B$ for each $\alpha_i\in\alpha$. For each $M\in\N$, we write $\alpha_M=\bigvee_{i=0}^{M-1}T_a^{-i}\alpha=\left\{\alpha_{M,1},\dots,\alpha_{M,k_M}\right\},k_M=\left|\alpha_M\right|$ and take a finite open cover $\beta_M=\left\{\beta_{M,1},\dots,\beta_{M,k_M}\right\}$ of $\T$ such that $\alpha_{M,i}\subset\beta_{M,i}$ and $\beta_{M,i}\prec T_a^{-l}B\ (0\leq l<M)$ for each $i=1,\dots,k_M$.
		Let $0<t<\min\{\log a,\log b\}$.
		If we show $K_{t^2/\log b}\subset Q\left(t,\left\{\beta_M\right\}_{M\in\N}\right)$, then, by Lemma \ref{keylemma}, we have $\dim_H K_{t^2/\log b}\leq \dim_H Q\left(t,\left\{\beta_M\right\}_{M\in\N}\right)\leq 2t/(\log a+t)$ and, by putting $t'=t^2/\log b$, obtain the theorem. We will show 
		$K_{t^2/\log b}\subset Q\left(t,\left\{\beta_M\right\}_{M\in\N}\right)$.
		
		Let $x\in K_{t^2/\log b}$ and take $\mu\in M_{\times a,\times b}(\T)$ such that $h_{\mu}(T_a)\leq t^2/\log b$ and $\delta_{\times a,\times b,x}^N\ (N\in\N)$ accumulates to $\mu$. We take a divergent subsequence $\left\{N_j\right\}_{j=1}^{\infty}$ in $\N$ such that $\delta_{\times a,\times b,x}^{N_j}\to\mu$ as $j\to\infty$.
		We take $0<\varepsilon<1$. Since $h_{\mu}(T_a,\alpha)=\lim_{M\to\infty}M^{-1}H_{\mu}\left(\alpha_M\right)\leq h_{\mu}(T_a)\leq t^2/\log b$,
		we have
		$$\frac{1}{M}H_{\mu}(\alpha_M)<\frac{t^2}{\log b}+\frac{t\varepsilon}{\log b}$$
		for sufficiently large $M\in\N$. We fix such an $M$.
		
		We write $q(\mu,\alpha_M)=\left(\mu\left(\alpha_{M,1}\right),\dots,\mu\left(\alpha_{M,k_M}\right)\right)$: a $k_M$-distribution and notice that\\
		$H\left(q(\mu,\alpha_M)\right)=H_{\mu}\left(\alpha_M\right)<M(t^2/\log b+t\varepsilon/\log b)$. We take a sufficiently small $\eta>0$ so that, for a $k_M$-distribution $q$, 
		\begin{equation}\label{eta}
			\left|q-q(\mu,\alpha_M)\right|<\eta \quad\text{ implies }\quad H(q)<M\left(\frac{t^2}{\log b}+\frac{t\varepsilon}{\log b}\right),
		\end{equation}
		where $|\cdot|$ denotes the Euclidean norm on $\R^{k_M}$. For each $i=1,\dots,k_M$, we take a compact subset $C_i$ such that $C_i\subset\alpha_{M,i}$ and $\mu\left(\alpha_{M,i}\setminus C_i\right)<\eta/2\sqrt{k_M}k_M$. Then we take an open subset $V_i$ such that $C_i\subset V_i\subset \beta_{M,i}$ and $V_i\ (i=1,\dots,k_M)$ are pairwise disjoint. Since $\delta_{\times a,\times b,x}^{N_j}\to\mu$ as $j\to\infty$ with respect to the weak* topology, we have
		\begin{equation*}
			\delta_{T_a,T_b,x}^{N_j}\left(V_i\right)>\mu(C_i)-\frac{\eta}{2\sqrt{k_M}k_M},\quad i=1,\dots,k_M,
		\end{equation*}
		hence
		\begin{equation*}
			\delta_{T_a,T_b,x}^{N_j}\left(V_i\right)>\mu(\alpha_{M,i})-\frac{\eta}{\sqrt{k_M}k_M},\quad i=1,\dots,k_M
		\end{equation*}
		for suffiently large $j$.
		
		For $j$ as above, we take a $N_j$-choice $\underline{\beta_M}=(\beta_{i_{m,n}})_{0\leq m,n<N}$ for $x$ with respect to $T_a,T_b$ and $\beta_M$ such that $i_{m,n}=i$ whenever $T_a^mT_b^nx\in V_i$. Then, when we write $q(\underline{\beta_M})=\left(q_1,\dots,q_{{k_M}}\right)$, we have
		\begin{equation*}
			q_i\geq\delta_{T_a,T_b,x}^{N_j}(V_i)>\mu(\alpha_{M,i})-\frac{\eta}{\sqrt{k_M}k_M},\quad i=1,\dots,k_M.
		\end{equation*}
		Since $q(\underline{\beta_M})$ and $q(\mu,\alpha_M)=\left(\mu\left(\alpha_{M,1}\right),\dots,\mu\left(\alpha_{M,k_M}\right)\right)$ are $k_M$-distributions, this implies that
		\begin{equation*}
			\left|q_i-\mu(\alpha_{M,i})\right|<\frac{\eta}{\sqrt{k_M}},\quad i=1,\dots,k_M.
		\end{equation*}
		Hence, by (\ref{eta}), we have
		\begin{equation}\label{TaTb}
			H(q(\underline{\beta_M}))<M\left(\frac{t^2}{\log b}+\frac{t\varepsilon}{\log b}\right).
		\end{equation}
		Now, since $0<t< \log b$,
		\begin{align*}
			q(\underline{\beta_M})=&\frac{1}{N_j}\sum_{n=0}^{N_j-1}q(\underline{\beta_M}_n)\\
			&=\frac{\lfloor tN_j/\log b\rfloor+1}{N_j}\left\{\frac{1}{\lfloor tN_j/\log b\rfloor+1}\sum_{0\leq n< tN_j/\log b}q(\underline{\beta_M}_n)\right\}\\
			&\qquad\qquad\qquad+\frac{N_j-\lfloor tN_j/\log b \rfloor -1}{N_j}\left\{\frac{1}{N_j-\lfloor  tN_j/\log b\rfloor-1}\sum_{tN_j/\log b\leq n<N_j}q(\underline{\beta_M}_n)\right\}.
		\end{align*}
		Hence, by the concavity of $H(p)$ in a $k_M$-distribution $p$ and (\ref{TaTb}), we have
		\begin{align*}
			&\frac{t}{\log b}H\left(\frac{1}{\lfloor tN_j/\log b \rfloor+1}\sum_{0\leq n< tN_j/\log b}q(\underline{\beta_M}_n)\right)\\
			\leq&\ \frac{\lfloor tN_j/\log b\rfloor+1}{N_j}H\left(\frac{1}{\lfloor tN_j/\log b\rfloor+1}\sum_{0\leq n<tN_j/\log b}q(\underline{\beta_M}_n)\right)\\
			&\qquad\qquad\qquad+\frac{N_j-\lfloor tN_j/\log b\rfloor -1}{N_j}H\left(\frac{1}{N_j-\lfloor tN_j/\log b\rfloor-1}\sum_{tN_j/\log b\leq n<N_j}q(\underline{\beta_M}_n)\right)\\
			\leq&\ H(q(\underline{\beta_M}))\\
			<&\ M\left(\frac{t^2}{\log b}+\frac{t\varepsilon}{\log b}\right)
		\end{align*}
		and
		\begin{equation*}
			H\left(\frac{1}{\lfloor tN_j/\log b \rfloor+1}\sum_{0\leq n< tN_j/\log b}q(\underline{\beta_M}_n)\right)
			< M(t+\varepsilon).
		\end{equation*}
		Using the concavity of $H(p)$ again, we have
		\begin{equation*}
			H(q(\underline{\beta_M}_n))< M(t+\varepsilon)
		\end{equation*}
		for some $0\leq n<tN_j/\log b$
		\footnote{In the last four expressions, we assume that $tN_j/\log b$ is not an integer. If $tN_j/\log b$ is an integer, then we replace $\lfloor tN_j/\log b\rfloor+1$ to $\lfloor tN_j/\log b\rfloor$ in these expressions and obtain the same conclusion.}.
		Since $q(\underline{\beta_M}_n)\in\Dist_{\beta_M}(T_b^nx,N_j)$, this shows that $x$ satisfies the condition in Lemma \ref{keylemma} for $N_j$ and $M$. Since this is satisfied for infinitely many $N_j\ (j\in\N)$ and sufficiently large $M\in\N$ and for arbitrary $0<\varepsilon<1$, we have $x\in Q\left(t,\left\{\beta_M\right\}_{M\in\N}\right)$. Then we have $K_{t^2/\log b}\subset Q\left(t,\left\{\beta_M\right\}_{M\in\N}\right)$ and complete the proof.
	\end{proof}
	
	\section*{Acknowledgements}
	The author is grateful to Masayuki Asaoka for his useful comments. He also thanks to Mitsuhiro Shishikura for his helpful advice. This work is supported by JST SPRING, Grant Number JPMJSP2110.

\end{document}